\documentclass[11pt,a4paper]{amsart}
\usepackage{amsmath,amsfonts}
\usepackage{graphicx}
\usepackage{amssymb}
\usepackage{amsthm}
\usepackage{yfonts}
\usepackage{kpfonts}
\usepackage{tikz-cd}
\usepackage{stmaryrd}
\usepackage{thm-restate}
\usepackage[percent]{overpic}

\usepackage{thmtools}
\usepackage{thm-restate}

\usepackage{hyperref}

\usepackage{mathtools}
\mathtoolsset{showonlyrefs}

\numberwithin{equation}{section}

\usepackage[usenames,dvipsnames]{pstricks}
\usepackage{epsfig}
\usepackage{rotating}
\usepackage{pst-plot}
\usepackage{pst-eps}
\usepackage{pst-grad}
\usepackage{pstricks-add}
\usepackage{lmodern}
\usepackage{xcolor}
\usepackage{graphicx}
\usepackage[top=3cm,bottom=3cm,left=3cm,right=3cm,a4paper]{geometry}

\usepackage{cite}

\newcommand{\PP}{{\mathbb P}}

\newcommand{\Pic}{{\rm{Pic}}}

\newcommand{\Effbar}{\overline{\rm{Eff}}}

\newcommand{\res}{\operatorname{res}}
\newcommand{\xdashrightarrow}[2][]{\ext@arrow 0359\rightarrowfill@@{#1}{#2}}

\newcommand{\Eff}{{\operatorname{Eff}}}

\newcommand{\M}[2]{\mathcal{M}_{{#1}, {#2}}}
\newcommand{\Mbar}[2]{\overline{\mathcal{M}}_{{#1}, {#2}}}
\newcommand{\Mbarmu}[2]{\Xi\overline{\mathcal{M}}_{{#1}, {#2}}(\mu)}

\newcommand{\QQ}{\mathbb{Q}}

\newcommand{\HH}{\mathcal{H}}

\newcommand{\F}{\mathcal{F}}
\newcommand{\RR}{\mathbb{R}}
\newcommand{\CC}{\mathbb{C}}
\newcommand{\AAA}{\mathbb{A}}

\newcommand{\OO}{\mathcal{O}}

\newcommand{\ZZ}{\mathbb{Z}}

\newcommand{\PGL}{\operatorname{PGL }}

\newcommand{\movcurve}{\overline{\mbox{Mov}}}

\newcommand{\Mgb}{\overline{\mathcal{M}}_g}
\newcommand{\Mgnb}{\overline{\mathcal{M}}_{g,n}}

\newcommand{\Cl}{\operatorname{Cl}}

\newcommand{\Num}{\operatorname{N}}
\newcommand{\Nef}{\operatorname{Nef}}

\newcommand{\Plus}{\texttt{+}}

\newcommand{\exc}{\mathfrak{exc}}

\theoremstyle{plain}
\newtheorem{theorem}{Theorem}[section]
\newtheorem{lemma}[theorem]{Lemma}
\newtheorem{proposition}[theorem]{Proposition}

\theoremstyle{definition}
\newtheorem{remark}[theorem]{Remark}
\newtheorem{definition}[theorem]{Definition}

\title[Isoresidual fibrations and moduli space]{Isoresidual Fibrations and the birational geometry of moduli of pointed stable curves}

\author[S. Mullane]{Scott Mullane}
\address{School of Mathematics and Statistics, University of Melbourne, Australia}
\email{{\tt mullanes@unimelb.edu.au}}

\begin{document}

\maketitle


\nopagebreak

\begin{abstract}
We show that the pseudo-effective cone of divisors of $\overline{\mathcal{M}}_{0,n}$ is not polyhedral for $n\geq 8$ by constructing an extremal non-polyhedral ray of the dual cone of moving curves via maps on meromorphic strata of differentials returning the residues at the poles of the parameterised differentials. An immediate corollary is that these spaces are not Mori Dream Spaces.
\end{abstract}

\section{Introduction}
Affine coordinates for the moduli space of isomorphism classes of pointed rational curves, $\mathcal{M}_{0,n}$, are obtained by normalising the $\PGL(2)$-action. Though this is only the start of the coordinate story. Kapravov described the Grothendieck-Knudsen compactification $\Mbar{0}{n}$ via Pl\"ucker coordinates to be a Chow quotient of the Grassmannian $G(2,n)$~\cite{Kap1}, and then as an iterative blow-up of projective $(n-3)$-space~\cite{Kap2}, a description amenable to birational geometry. While hypergeometric functions are used to obtain the Aomoto-Gelfand or GKZ coordinates~\cite{A87,GKZ90} that help understand the monodromy of the Gauss-Manin connection, and Cluster algebra or Fock-Goncharov coordinates \cite{FG06} are obtained from diagonals of $n$-gons, or cyclic dihedral cross-ratios and have had significant applications in scattering amplitudes in physics.{{}}

Alternatively, given any integer partition $\mu=(m_1,\dots,m_n)$ of -$2$, a pointed rational curve admits a unique differential with zeros and poles at the marked points of type $\mu$ up to $\CC^*$-scaling. Integrating this differential on a $\ZZ$-basis for homology of the curve punctured at the poles ($p_i$ corresponding to negative $m_i$) relative to the zeros ($p_i$ corresponding to non-negative $m_i$) gives, courtesy of the Gauss-Manin connection, 
manifold coordinates for $\mathcal{M}_{0,n}$ with transition functions obtained as change of basis matrices for the relative homology.  In higher genus, through a connection with dynamics, these transcendental coordinates have revealed an unexpected algebraic structure~\cite{EskinMirzakhani,EskinMirzakhaniMohammadi,Filip}. In this paper, we bring this perspective to suborn questions on the birational geometry of the Grothendieck-Knudsen compactification $\Mbar{0}{n}$. {{}}

The pseudo-effective cone of an irreducible $\QQ$-factorial projective variety is the closure of the cone of effective divisors and prescribes the essential birational information of the variety. 
The study of the effective cone of moduli spaces of curves has attracted sustained attention dating back to pioneering work of Harris and Mumford~\cite{HarrisMumford}, who constructed effective divisors to show that $\Mgb$ is of maximal Kodaira dimension for $g$ large enough. 
Over the last 25 years, understanding the pseudo-effective and nef cones (closure of the subcone of ample divisors) of $\Mbar{0}{n}$ has proven to be a difficult question. Though the Kapranov construction gives an explicit description of $\Mbar{0}{n}$ as an iterative blow-up of $\PP^{n-3}$, the situation is complicated by the Picard rank growing exponentially in $n$. Numerous efforts have been undertaken to understand the effective and ample divisors from a number of perspectives~\cite{C09,Faber, F20,GG12,GM11,HK00,CastravetTevelevHypertree,CastravetTevelev,CLTU, T25}. 
{{}}

The structure of the nef cone, a question originally posed by Mumford, has a conjectural description known as the \emph{F-conjecture}~\cite{KMc96} named in honour of Fulton, who first posed the conjecture, and Faber, who proved early cases. In genus zero, the conjecture can be stated as: the Mori cone of curves is generated by the irreducible components of the one-dimensional boundary loci, that is, the loci of stable curves with at least $n-4$ nodes. Gibney, Keel and Morrison~\cite{GKM} showed that the conjecture in genus zero would imply the result in higher genus, concentrating attention in the genus zero case. 
The conjecture is currently known to hold for $n\leq8$~\cite{KMc96,FM}, leaving open $n\geq 9$, with some notable attempts to construct counter examples~\cite{Chen11,CT12}
{{}}

The effective cone of $\Mbar{0}{5}$, 
is generated by the irreducible components of the boundary $\partial\Mbar{0}{n}=\Mbar{0}{n}\setminus\mathcal{M}_{0,n}$, though Hassett and Tschinkel~\cite{HT} showed the effective cone of $\Mbar{0}{6}$ is generated by the boundary and Keel-Vermiere\footnote{Originally provided as a counter example to Fulton's first conjecture that cones of effective cycles in $\Mbar{0}{n}$ in all codimension were generated by boundary classes.} divisors~\cite{KV} (one $S_6$-orbit) that intersect the interior of the moduli space. Castravet and Tevelev~\cite{CastravetTevelevHypertree} generalised the Keel-Vermiere divisors through Brill-Noether theory on singular higher genus curves to produce finitely many extremal rays in the effective cone of $\Mbar{0}{n}$ for each $n\geq 7$ indexed by irreducible hypertrees and conjectured that  these rays with the boundary divisors generated the effective cone. Opie~\cite{Opie} found extremal rays that contradicted the conjecture and further extremal rays were found by Doran,
Giansiracusa, and Jensen~\cite{DGJ}. More work has followed in the case of $n=7$, where Dutour Sikiri\'c and Jovinelly~\cite{DSJ25} found $37$ $S_7$-orbits of new extremal rays via numerical methods and He and Yang~\cite{HY21} showed that the pseudo-effective cone of $\Mbar{0}{7}$ blown-up at a general point is not polyhedral. {{}}

Hu and Keel~\cite[Question 3.2]{HK00} asked whether $\Mbar{0}{n}$ is a Mori Dream Space (MDS) which is equivalent to asking if the Cox ring is finitely generated~\cite[Prop 2.9]{HK00}. 
 This implies the effective cone is polyhedral, but further, it results in the variety being the simplest possible from the perspective of the minimal model program, ``Mori's dream". The effective cone decomposes into finitely many chambers each giving different birational models for the moduli space with potentially modular meaning. In a breakthrough result, Castravet and Tevelev~\cite{CastravetTevelev} used blow-ups of toric surfaces 
  to show that $\Mbar{0}{n}$ has non-finitely generated Cox ring and hence is not a MDS for $n \geq134$. This strategy was used by Gonzalez and Karu to lower the bound to $n\geq 13$~\cite{GK} and then by Hausen, Keicher, and Laface to $n\geq 10$~\cite{HKL}. This approach was then extended by Castravet, Laface, Tevelev, and Ugaglia~\cite{CLTU} to show the stronger result that the effective cone of $\Mbar{0}{n}$ is non-polyhedral for $n\geq10$. However, as explained in~\cite{HKL}, this strategy has been optimised. New ideas are required to approach both the MDS question and the question of polyhedrality of the pseudo-effective cone in the remaining open cases $n=7,8,9$.{{}}

For $g\geq 1$ and $\mu$ any non-zero length $n$ integer partition of $2g-2$, the stratum of canonical divisors of type $\mu$ form a substacks of $\mathcal{M}_{g,n}$ which have been used to identify infinite families of extremal effective divisors~\cite{ChenCoskun, MEff} and higher codimension cycles~\cite{Mhigher} in $\Mgnb$ for fixed $g,n$ and hence show these cones to be non-polyhedral. Recently, in the culmination of much work, a smooth compactifcation of these spaces via multi-scale differentials has been achieved~\cite{BCGGM}  (presented in~\S\ref{sec:multi-scale}). However, when $g=0$ the stratum of canonical divisors of type $\mu$ is identified with $\mathcal{M}_{0,n}$. Hence the strata no longer identify interesting collections of subvarieties and we require a new approach. 

Let $\mu=(a_1,\dots,a_m,-b_1,\dots,-b_{n-m})$ be a partition of -$2$ with $a_i, b_i$ positive integers. A differential of type $\mu$ on any smooth $n$-pointed rational curve is uniquely determined up to scaling by $\CC^*$. Hence for each choice of $\mu$ we obtain a residue map
$$r_\mu: \M{0}{n}\dashrightarrow \PP^{n-m-2}$$
which returns the residues at the first $n-m-1$ poles up to $\CC^*$-scaling (the residue at the final pole is determined by the others). For $m=1$, the map $r_\mu$ becomes finite onto its image and a number of recent works \cite{GT21,GT22,CP23} have investigated this case via flat and algebro-geometric methods. The map is ramified over the resonance hyperplane arrangement in $\PP^{n-m-2}$, defined as the collection of hyperplanes $H_S$ cut out by $\sum_{j\in S}X_j$ for all subsets $S\subset \{0,\dots,n-m-2\}$. The authors then combined in \cite{CGPT24} to consider the case $m=2$ where $r_\mu$ has one dimensional fibres. The generic fibre is a complex curve, the closure of which, in the multi-scale compactification inherits a canonical translation structure. This induces a vector bundle over the complement of the resonance hyperplane arrangement endowed with a connection that acts trivially locally on the fibres, the well-known \emph{Gauss-Manin connection}. {{}}

The study of residues on differentials on rational curves has also arisen recently from a number of other perspectives including the dynamics of polynomial maps~\cite{Sug17}, the Kadomtsev-Petviashvili hierarchy~\cite{BR24}, and the topology of configuration spaces~\cite{Sal23}. In this paper we turn this perspective to the birational geometry of moduli space.{{}}

Our approach is to identify a non-polyhedral extremal ray of the cone of moving curves $\movcurve(\Mbar{0}{n})$ dual to the pseudo-effective cone of divisors.  
The following theorem shifts our search further, from extremal rays of the moving cone, to a search for rational relative dimension one maps with projective general fibre.

\begin{restatable}{theorem}{exrat}\label{thm:exrat}
Let $\pi:X\dashrightarrow Y$ be a dominant rational map between an irreducible normal $\QQ$-factorial projective variety $X$ and a smooth projective variety $Y$ of relative dimension one such that $\pi$ restricts to a morphism over some Zariski open subset $U$ of $Y$. Then the numerical class of a general fibre of $\pi$ forms an extremal ray of $\movcurve(X)$, the cone of moving curves on $X$.
\end{restatable}

Residue maps offer a number of cases where the map extends to a morphism over the complement of the resonance hyperplane arrangement.

\begin{restatable}{theorem}{Fmuextremal}\label{thm:Fmuextremal}
Let $\mu=(a,n-4-a,(-1)^{n-2})$ for $1\leq a\leq n-5$. Then the residue map
$$r_\mu:\Mbar{0}{n}\dashrightarrow \PP^{n-4}$$
has projective general fibre, the numerical class of which forms an extremal moving curve denoted by $F_\mu$. 
\end{restatable}

To identify a non-polyhedral extremal ray, we restrict to a specific signature and compute the rank (Definition~\ref{defn:rank}) of the moving curve explicitly. The forgetful morphism implies that non-polyhedrality for $\Mbar{0}{n}$ implies non-polyhedrality for all $\Mbar{0}{n'}$ with $n'\geq n$ 
and it can be shown the extremal moving curves identified here are polyhedral for $n=7$. Hence we restrict to $n=8$ and let $\mu=(2^2,-1^6)$. 

The proof Theorem~\ref{thm:Fmuextremal} actually relies only on the Hodge bundle compactification which agrees with the multi-scale compactification over the complement of the resonance hyperplane arrangement. However, from here the utility of the new multi-scale compactification comes into full view as it resolves the residue map over the full boundary and streamlines what would have been a particularly tedious task of computing the rank of the moving curve of interest. By Lemma~\ref{lem:resolve} the multi-scale compactification resolves the residue map $r_\mu$ as
$$(\phi,{\pi}): \Mbarmu{0}{8}\longrightarrow \Mbar{0}{8}\times \PP^{4}. $$
We begin by computing the rank of the curve class $\widetilde{F}_\mu$ general fibre of $\pi$ which projects to give $F_\mu=\phi_*\widetilde{F}_\mu$.  That is, the rank of the vector space
$$\RR\otimes\{D\in \Effbar(\Mbarmu{0}{8}) |\pi_*D=0  \}.$$
Debarre, Jiang, and Voisin have shown~\cite[Theorem 5.1]{DJV13} that for surjective morphisms, the $\RR$-vector spaces generated by effective and pseudo-effective cycles in $\ker(\pi_*)$ coincide for curves and divisors and further conjecture the validity of the analogous statement in all other dimensions. In \S\ref{sec:Stein} we compute the Stein factorisation of $\pi$, while in Proposition~\ref{prop:rankup} we carefully compute the rank of the curve $\widetilde{F}_\mu$. This is done by considering first the effective divisors supported on the pullback of the resonance hyperplane arrangement, then using the Stein factorisation to extend the computation across $\Mbarmu{0}{8}$. Pushing this result forward to $\Mbar{0}{8}$ we obtain the following.

\begin{restatable}{theorem}{corank}
\label{thm:corank}
Let $\mu=(2^2,-1^6)$, the rank of $\RR\otimes\{D\in \overline{\Eff}(\Mbar{0}{8}) |F_\mu\cdot D=0  \}$ is $\rho(\Mbar{0}{8})-6.$
\end{restatable}

Hence the extremal ray $F_\mu$ of the cone of moving cures is not polyhedral and dually:
\begin{restatable}{theorem}{MOn}
\label{thm:nonpoly}
$\Effbar(\Mbar{0}{n})$ is not polyhedral for $n\geq 8$.
\end{restatable}

This result has the immediate corollary.

\begin{restatable}{corollary}{MDS}
\label{cor:MDS}
$\Mbar{0}{n}$ is not a Mori Dream Space for $n\geq 8$.
\end{restatable}

We note here that this range is exactly the range where the anticanonical divisor is not pseudo-effective and with the recent result~\cite{FM} now includes the first example, $n=8$, where the F-conjectures is known to hold, and the pseudo-effective cone is not polyhedral (and hence the MDS property also fails).{{}}

The above constructions have a side effect of identifying a number of interesting divisor classes we refer to as the \emph{resonance transform divisors}. For $\mu=(a_1,\dots,a_m,(-1)^{n-m})$ with $a_i>0$ and $\sum a_i=n-m-2$, in Definition~\ref{def:resdiv} we define for any $S\subset\{0,\dots,n-m-2\}$, the resonance transform divisor as $D^\mu_S:=\phi_*\widetilde{H}_S$ in $\Mbar{0}{n}$ where $\widetilde{H}_S$ in $\Mbarmu{0}{n}$ defined as the components of $\pi^*H_S$ that intersect the interior of the moduli space. In the final section of the paper we compute the class of these divisors in Proposition~\ref{prop:divclass}, and prove rigidity and extremality via an inductive argument on $m$ with the base case of $m=2$ proven by the construction of covering curves of the divisor with negative intersection.

\begin{restatable}{theorem}{extremaldivisors}
\label{prop:extremaldivisors}
For $\mu=(a_1,\dots,a_m,(-1)^{n-m})$ with $m\geq2$, $a_i>0$ and $\sum a_i=n-m-2$, the resonance transform divisor $D_S^\mu$ is rigid and extremal in $\Mbar{0}{n}$.
\end{restatable}

In \S\ref{sec:flat} we provide the background on flat geometry and the strata of differentials including the residue map. In \S\ref{sec:biratgeom} we introduce and prove the required results on extremal moving curves, morphisms, and birational geometry. In \S\ref{sec:resolve} we introduce the multi-scale compactification that provides a resolution of the residue maps of interest and provide the Stein factorisation of the resolved map in the case that $\mu=(2^2,-1^6)$. In \S\ref{sec:effcone} we provide proofs of all main theorems.
\\
\\
\textbf{Acknowledgements.} I am very grateful to Fred Benirschke, Brian Lehmann, and Dawei Chen for stimulating discussions on topics related to this paper. This project was supported by DECRA Grant DE220100918 from the Australian Research Council.

\section{Flat geometry and the strata of differentials}\label{sec:flat}
The stratum of differentials $\HH(\mu)$ for $\mu=(m_1,\dots,m_n)\in \ZZ^n$ with $\sum m_i=-2$, parameterises pairs $([C,p_1,\dots,p_n],\eta)$ of a smooth $n$-pointed rational curve and a non-zero differential $\eta$ such that $\text{div}(\eta)=\sum_{i=1}^nm_ip_i$. The scheme or stack structure of $\HH(\mu)$ is inhertited as a substack of the appropriately twisted Hodge bundle and the dimension of $\HH(\mu)$ is $n-2$. Clearly, the stratum $\HH(\mu)$ is simply a $\CC^*$-bundle over $\mathcal{M}_{0,n}$
 with fibres obtained by scaling the differental. However, the utility of this perspective is in the consideration of \emph{period coordinates} that provide local orbifold coordinates for the strata and hence also $\mathcal{M}_{0,n}$. {{}}

The differential $\eta$ induces a flat metric on the Riemann surface punctures at the poles $P=\{p_i\mid m_i<0\}$ with conical singularities of angle $2\pi(m_i+1)$ at the zeros $Z=\{p_i\mid m_i\geq 0\}$. Further, considering a basis for the relative homology $\gamma_1,\dots,\gamma_{n-2}\in H_1(C\setminus P,Z,\ZZ)$ we obtain local complex analytic coordinates for $\HH(\mu)$ by integration of $\eta$
\begin{eqnarray}  
&\HH(\mu)\longrightarrow\AAA^{n-2} &\\
&([C,p_1,\dots,p_n],\eta)\mapsto \left(\int_{\gamma_{1}}\eta,\dots,\int_{\gamma_{n-2}}\eta\right)&
\end{eqnarray}

The residues at the poles form a subset of period coordinates that are defined globally. The residue map denoted 
$$r_\mu:\HH(\mu)\longrightarrow \mathcal{R}\subset \AAA^r$$ 
where $r$ is the number of $m_i<0$, returns the residues of $\eta$ at the $p_i$ such that $m_i<0$. {{}}

The image $\mathcal{R}$ of $r_\mu$ for all $\mu$ was determined in \cite{GT21}. The \emph{resonance hyperplane} $H_S$ in $\AAA^r$ for any nonempty $S\subsetneqq[r]$ is defined as 
$$H_S:=\{(x_1,\dots,x_r)\in\mathcal{R}\mid \sum_{j\in S}x_i=0 \}$$
and define the union over all such $S$ to be the \emph{resonance hyperplane arrangement} in $\mathcal{R}$.{{}}

When only one $m_i$ is non-negative, $r_\mu$ becomes finite onto its image and a number of recent works \cite{GT21,GT22,CP23} have investigated this case via flat and algebro-geometric methods. The map is finite and is ramified over the resonance hyperplane arrangement. The authors then combined in \cite{CGPT24} to consider the case where two $m_i$ are non-negative and $r_\mu$ has one dimensional fibres. The generic fibre $F_\lambda$ is a complex curve with closure $\overline{F}_\lambda$ in the multi-scale compactification (presented in~\S\ref{sec:multi-scale}), a possibly disconnected compact curve. In any such fibre, the relative period between the zeros, say $z$, gives a differential $dz$ that as the residues are fixed, is independent of the choice of basis and hence establishes a canonical translation structure $\omega_\lambda$ on the fibre $\overline{F}_\lambda$. This induces a vector bundle with fibres   $\CC\otimes H_1(\overline{\F}_\lambda\setminus P_{\omega_\lambda},Z_{\omega_\lambda},\ZZ)$ over $\lambda\in \mathcal{R}$ endowed with a connection over the residue space that acts trivially locally on the fibres known as the \emph{Gauss-Manin connection}. The authors provide a classification of the connected components of the generic fibre of $r_\mu$ when fibres are positive dimensional. The invariants are related to the topological invariants of strata of translation and dilation surfaces of higher genus obtained via a certain consistent surgery on the flat surfaces. 

\begin{theorem}[Theorem 1.7 \cite{CGPT24}]\label{thm:conncomp}
The generic fibres of $r_\mu$ are connected except for the following two families of strata. Let $a_1,\dots,a_m,b_1,\dots, b_p$ be positive integers with $m\geq 2$, then
\begin{enumerate}
\item
$\HH (ka_1,\dots,ka_m,-kb_1,\dots,-kb_{p-2},(-1)^2)$ for some $k \geq 2$ and $a_1,\dots,a_n,b_1,\dots,b_{p-2}$ positive coprime integers, in which case the generic fibres have $k$ connected components;
\item
$\HH (2a_1,\dots,2a_m,-2b_1,\dots,-2b_{p-2t},(-1)^{2t})$ with $t\geq  2$, where generic fibres have two connected components.
\end{enumerate}
\end{theorem}

By observing that the final residue is determined by the first $r-1$, the map descends  modulo the action of $\CC^*$ to a map 
$$r_\mu:\HH(\mu)/\CC^*\cong \M{0}{n}\dashrightarrow \PP^{r-2}.$$
Further, if $\underline{0}\notin\mathcal{R}$ then $r_\mu$ is well-defined everywhere. {{}}

Similarly to the unprojectivised case, for any non-empty $S\subset \{0,1,\dots,r-2\}$ we define the \emph{resonance hyperplane} $H_S$ by the equation $\sum_{j\in S}X_j$ and their union
$$\mathcal{A}:=\sum_{S\subset\{0,\dots,r-2\}}H_S$$
to be the \emph{resonance hyperplane arrangement} in $\PP^{r-2}$. Away from this locus, $r_\mu$ has been described in the discussion above. This map does not extend to a morphism over the Grothendieck-Knudsen compactification $\Mbar{0}{n}$. In the \S\ref{sec:resolve} we resolve this extension in the cases of interest to us, and determine the Stein factorisation of the resulting morphism for $\mu=(2^2,-1^6)$. It is the investigation of the rich structure of this map that provides the results of this paper.

\section{Morphisms and birational geometry}\label{sec:biratgeom}
We recall the following well known definitions (See for example, \cite{Laz}). Let $X$ be a normal projective irreducible variety, $\Cl(X)$ the divisor class group, and $\Pic(X)$ the Picard group of $X$. Let 
$\equiv$ denote numerical equivalence and denote the numerical equivalence classes of Cartier divisors by
$$\Num^1(X):=\Pic(X)\otimes\RR/\equiv  .$$ 
We similarly define $\Num_1(X)$ to be numerical equivalence classes of $\RR$-linear combinations of irreducible curves. It follows that the intersection product on these vector spaces is a perfect pairing and hence $\Num^1(X)$ and $\Num_1(X)$ form dual finite dimensional vector spaces.{{}}

We define the effective cone of divisors $\Eff(X)\subset \Num^1(X)$ to be the convex cone generated by numerical classes of effective Cartier divisors, and the pseudo-effective cone $\Effbar(X)$ to be the closure of this cone. We define $\Nef^1(X)\subset\Num^1(X)$ to be the cone generated by nef divisor classes and the moving cone of curves $\movcurve(X)\subset \Num_1(X)$ to be the closure of the cone generated by moving curve classes. The cones $\Effbar(X)$ and $\movcurve(X)$ are dual to each other (see~\cite{BDPP} for the smooth case, \cite[Thm 11.4.19]{Laz} in general).

An element of a convex cone is known as \emph{extremal} if it cannot be expressed as a sum of two non-proportional elements of the cone. An effective divisor $D$ is \emph{rigid} if $\dim H^0(D)=1$ and $\dim H^0(kD)=1$ for all positive integers $k$ and an irreducible rigid divisor $D$ is extremal in $\Eff(X)$. To the author's knowledge, it is not known if there exist irreducible rigid effective divisors that are not extremal in $\Effbar(X)$.

\begin{definition}\label{defn:polyhedral}
A convex cone $\mathcal{C}\subset\RR^m$ is known as \emph{polyhedral} if there are finitely many vectors $v_1,\dots, v_s\in\RR^m$ such that $\mathcal{C}=\RR_{\geq0}v_1\cdots \RR_{\geq0}v_s$. The cone is known as \emph{simplicial} if further $s=m$.
\end{definition}

The Minkowski-Weyl theorem for polyhedra simplifies in the case of cones stating that a polyhedral convex cone can be equivalently described by finitely many codimension one faces, or half-planes, corresponding to the finitely many extremal rays of the dual cone. We translate this to our setting as follows.

\begin{definition}\label{defn:rank}
The \emph{rank} of a moving curve $F\in \movcurve(X)$ is defined as the rank of 
$$\RR\otimes\{D\in \Effbar(X)\mid F\cdot D=0\}. $$
\end{definition}

\begin{theorem}
If $\Effbar(X)$ is polyhedral, then every extremal moving curve $F$ has rank $\rho(X)-1$. 
\end{theorem}

The following theorem gives a criterion for identifying extremal moving curve classes with the potential to be non-polyhedral. We start with the case of a morphism.

\begin{theorem}\label{thm:extremal}
Let $\pi:X\longrightarrow Y$ be a surjective morphism between irreducible normal $\QQ$-factorial projective varieties with relative dimension one. Then the numerical class of the general fibre of $\pi$ spans an extremal ray of $\movcurve(X)$, the cone of moving curves on $X$.
\end{theorem}

\begin{proof}
Let the class of the general fibre be denoted $F$. Then $F$ is clearly a moving curve as effective curves with class equal to $F$ cover a Zariski dense subset of $X$. If $F$ has multiple components, then they must all have the same class in $N_1(X)$. If they can be distinguished by their class, taking the closure of the subset of curves distinguished in this way would provide multiple components of $X$ and provide a contradiction to irreducibility. {{}}

If $F$ is not extremal then there exists a decomposition $F\equiv F_1+F_2$ where $F_1,F_2$ are moving curves and $F_1\ne kF_2$ for any $k\in\RR$. Then as the classes are not proportional, there exists a (necessarily not effective) divisor $L$ such that $F\cdot L>0$ and $F_1\cdot L<0$.{{}}

Hence $L$ is $\pi$-big, and by \cite[Lemma 3.23]{KM98}, for any $A$ ample on $Y$, the divisor $D_m=L+m\pi^*A$ is big for $m\gg0$. However, $F\cdot \pi^*A=0$ and hence $F_1\cdot\pi^*A=F_2\cdot \pi^*A=0$ as each $F_i$ is a moving curve and $\pi^*A$ is effective implying $F_i\cdot\pi^*A\geq0$. Hence $F_1\cdot D_m=F_1\cdot L<0$ for all $m$, contradicting the assumption that $F_1$ is a moving curve.
\end{proof}

The case of use to us is the rational case where the map extends to a morphism over a Zariski open in the target.

\exrat*

\begin{proof}
Let the class of the closure of the general fibre be denoted $F$. There exists a resolution of $(\phi,\tilde{\pi}):Z\longrightarrow X\times Y$ of $\pi:X\dashrightarrow Y$ by  a normal $\QQ$-factorial variety $Z$ and the curve class $\widetilde{F}$ given by the general fibre of $\tilde{\pi}$ pushes forward to give $\phi_*\widetilde{F}=F$ in $N_1(X)$. Further, the curve $F$ is a moving curve, as curves with class equal to $F$ cover a Zariski dense subset of $X$. 

If $F$ is not extremal then there exists a decomposition $F\equiv F_1+F_2$ where $F_1,F_2$ are moving curves and $F_1\ne kF_2$ for any $k\in\RR$. Then as the classes are not proportional, there exists a (necessarily not effective) divisor $L$ such that $F\cdot L>0$ and $F_1\cdot L<0$.

Observe $\widetilde{F}\cdot\phi^*L=\phi_*\widetilde{F}\cdot L=F\cdot L>0$ and hence $\phi^*L$ is $\tilde{\pi}$-big. Hence by~\cite[Lemma 3.23]{KM98}, for any $A$ ample on $Y$, the divisor $D_m=\phi^*L+m\tilde{\pi}^*A$ is big for $m\gg0$. But then $\phi_*D_m$ is big as the push forward of a big divisor under a birational morphism.

Finally, we observe that $0=\widetilde{F}\cdot \tilde{\pi}^*A=\widetilde{F}\cdot \phi^*\phi_*\tilde{\pi}^*A$ as $\widetilde{F}$ has zero intersection with the exceptional locus of $\phi$. Hence 
$$0=\phi_*\widetilde{F}\cdot \phi_*\tilde{\pi}^*A=F\cdot  \phi_*\tilde{\pi}^*A$$
and as $\tilde{\pi}^*A$ is effective, this implies that also $0=F_1\cdot  \phi_*\tilde{\pi}^*A=F_2\cdot   \phi_*\tilde{\pi}^*A$ as both are assumed to be moving curves hence $F_i\cdot\phi_*\tilde{\pi}^*A\geq 0$ with sum equal to zero as $F\cdot\phi_*\tilde{\pi}^*A =0$. Hence
$$F_1\cdot\phi_*D_m=F_1\cdot\phi_*\phi^*L=F_1\cdot L<0$$
contradicting the assumption that $F_1$ is a moving curve.
\end{proof}

Debarre, Jiang, and Voisin conjectured~\cite{DJV13} that for a morphism of smooth varieties $\pi:X\longrightarrow Y$ the $\RR$-vector space generated by pseudo-effective $k$-dimensional cycles in $\ker(\pi_*)$ is equal to the $\RR$-vector space generated by effective $k$-dimensional cycles contracted by $\pi$ and proved the statement in the case of curves and divisors. We record the case of divisors as a theorem to complete this section.

\begin{theorem}[Theorem 5.1 \cite{DJV13}]\label{thm:DJV}
Let $\pi:X\longrightarrow Y$ be a surjective morphism between a normal $\QQ$-factorial projective variety $X$ and a smooth projective variety $Y$, then the $\RR$-vector space generated by pseudo-effective divisors in $\ker(\pi_*)$ is equal to the $\RR$-vector space generated by effective divisors contracted by $\pi$.
\end{theorem}

\begin{proof}
This is a slight enhancement of the result for $X$ smooth \cite[Theorem 5.1]{DJV13} that we will need. A proof is obtained by resolving the singularities in $X$ and the observation that under a birational morphism $f:Z\longrightarrow X$, a divisor $D\in N^1(X)$ is strictly pseudo-effective if and only if $f^*D\in N^1(Z)$ is strictly pseudo-effective.
\end{proof}

\section{Resolution and Stein factorisation of the residue map}\label{sec:resolve}
\subsection{Multi-scale compactifications $\Mbarmu{0}{n}$}\label{sec:multi-scale}
The residue maps $r_\mu$ described in \S\ref{sec:flat} do not extend to morphisms over the Grothendieck-Knudsen compactification. To resolve this, we use the multi-scale compactification of the strata of differentials introduced in~\cite{BCGGM}. The projectivisation provides a smooth Deligne-Mumford stack with normal crossings boundary we will denote by $\Mbarmu{0}{n}$ that admits a forgetful morphism $\phi:\Mbarmu{0}{n}\longrightarrow \Mbar{0}{n}$ which is an isomorphism over $\mathcal{M}_{0,n}$, and we will refer to $\phi^{-1}(\mathcal{M}_{0,n})$ as the \emph{interior}.

\begin{definition}
Let $\mu=(m_1,\dots,m_n)\in \ZZ^n$ with $\sum m_i=-2$. Let $\Gamma(z)$ be the dual graph associated to a stable curve $z\in\Mbar{0}{n}$ with vertices $V$, edges $E$ and half edges $H$ (associated to the $n$-marked points).  An \emph{enhanced level structure} $\widetilde{\Gamma}(z)$ on the dual graph $\Gamma(z)$ is defined as 
\begin{enumerate}
\item
A \emph{level structure} on the vertices $V$, that is, a surjective map $\ell:V\longrightarrow\{0,-1,\dots,-L\}$ for some non-negative integer $L$. 

Hence at each vertex $v\in V$, there is a partition of of the edges at $v$ as 
 $$E_v=E_v^+\sqcup E_v^-\sqcup E_v^o$$ where the subsets are the edges from $v$ to to vertex on higher, lower, or the same level respectively. 
\item
An \emph{enhancement} of the level graph, that is, a map $\kappa:E\longrightarrow \ZZ_{\geq 0}$ such that:
\begin{enumerate}
\item
 $ \kappa(e)=0$ for edges between vertices on the same level (\emph{horizontal edges}), 
 \item
 $\kappa(e)>0$ for edges between vertices on different levels (\emph{vertical edges}), 
 \item
 At each vertex $v$,  
$$\sum_{p_i\in H_v} m_i-|E_v^o|-\sum_{e\in E_v^+}(\kappa(e)+1)+\sum_{e\in E_v^-}(\kappa(e)-1)=-2,   $$
where $H_v$ denote the half-edges at $v$.
\end{enumerate}
\end{enumerate}
\end{definition}

\begin{definition}
Let $\mu=(m_1,\dots,m_n)\in \ZZ^n$ with $\sum m_i=-2$. A \emph{multi-scale differential} of type $\mu$ is
$$([C,p_1,\dots,p_n],\Gamma,\eta, \Psi)$$
that satisfies the following requirements:
\begin{enumerate}
\item
$[C,p_1,\dots,p_n]$ is a stable curve in $\Mbar{0}{n}$
\item
$\Gamma$ is an enhanced level structure compatible with $[C,p_1,\dots,p_n]$. 
\item 
$\eta=\{\eta_v\}_{\{v\in V\}}$ is a collection of differentials on the irreducible components of the curve such that the following are satisfied:
\begin{itemize}
\item \emph{Partitioning.}  The associated divisor for each $\eta_v$ satisfies
$$\text{div}(\eta_v)=\sum_{p_i\in H_v} m_ip_i -\sum_{p\in E_v^o} p-\sum_{p\in E_v^+}(\kappa(e)+1)p+\sum_{p\in E_v^-}(\kappa(e)-1)p.   $$ 
\item \emph{Residue matching.}
If two vertices $v_1$ and $v_2$ are connected in the level graph by a horizontal edge, then at associated node, $q$, we require
$$\res_q(\eta_{v_1}) +\res_q(\eta_{v_2})=0.    $$ 
\item \emph{Simultaneous scaling.} 
Collections $\eta=\{\eta_v\}_{\{v\in V\}}$ are considered up to equivalence under the action of the \emph{level rotation torus}, $(\CC^*)^{L}$, that acts by scaling differentials on each level below zero.  
\end{itemize}
\item
\emph{Prong matching} $\Psi$, consisting of a prong matching for each vertical edge that is finite data relating the differentials across the levels. We refer to \cite{BCGGM} for the details and discussion.
\end{enumerate}
A \emph{projectivised multi-scale differential} of type $\mu$ is a multi-scale differential of type $\mu$ considered up to the action of $\CC^*$ acting on the top level. We denote the smooth Deligne-Mumford stack of all such objects as $\Mbarmu{0}{n}$.

\end{definition}

\begin{remark}
The global residue condition (GRC) is empty in the case of $g=0$.
\end{remark}

\subsection{Cohomology of $\Mbarmu{0}{n}$}\
The rational Picard group $\Pic_\QQ(\Mbarmu{0}{n})$ has a basis given by the pullback of a basis for $\Pic_\QQ(\Mbar{0}{n})$ under $\phi$ and the irreducible components of the exceptional locus of $\phi:\Mbarmu{0}{n}\longrightarrow \Mbar{0}{n}$. These boundary divisors are indexed by enhanced level graphs with no horizontal edges, two levels, and at least three vertices.

\subsection{Resolving residue maps}
If all poles are simple, the multi-scale compactification resolves the residue map on $\Mbar{0}{n}$. We record this as a lemma.

\begin{lemma}\label{lem:resolve}
Let $\mu=(a_1,\dots,a_m,-1^{n-m})$ with $a_i>0$ and $\sum a_i=n-m-2$, then the multi-scale compactification resolves the residue map, that is
$$(\phi,{\pi}): \Mbarmu{0}{n}\longrightarrow \Mbar{0}{n}\times \PP^{n-m-2} $$
forms a resolution of 
$$r_\mu:\Mbar{0}{n}\dashrightarrow \PP^{n-m-2}.$$
\end{lemma}

\begin{proof}
 For any multi-scale differential of type $\mu$ there must be at least two poles appearing on top level in the associated enhanced level graph. Hence in these cases the residue map can be extended to a morphism over $\Mbarmu{0}{n}$ by returning the residues at the poles on top level and zero for the residue at any pole appearing below the top level. 
 \end{proof}

The following lemma will also be of use.

\begin{lemma}\label{lem:noexcept}
In the case that $\mu=(a_1,\dots,a_m,-1^{n-m})$ with $a_i> 0$, and $m=1,2,3$, no component of the exceptional locus of $\phi$ dominates $\mathbb{P}^{n-m-2}$ under the map $\pi$. 
\end{lemma}

\begin{proof}
Any enhanced level graph $\Gamma$ specifying an exceptional divisorial component has two levels, no horizontal edges, and at least three vertices. However, in order for the divisor to dominate $\mathbb{P}^{n-m-2}$ under the map $\phi$, we would require $\Gamma$ to have all poles appearing on top level with no imposed residue condition and hence appear on a unique vertex on top level.  {{}}

This leaves only three points $\{p_1,p_2,p_3\}$ to distribute across the remaining vertices on lower level. Each vertex requires at least two of these marked points to be stable providing a contradiction.
\end{proof}

The tautological class $\xi$ on $\Mbarmu{0}{n}$ for $\mu=(m_1,\dots,m_n)$ corresponding to the pullback of the dual of the hyperplane class $\mathcal{O}(1)$ on $\PP^{n-m-2}$ is given in~\cite[Prop 8.1]{CMZ22} by
\begin{equation}\label{eq:taut}
 \xi=(m_i+1)\psi_i-\sum_{\Gamma\in{}_iLG_1(\mathcal{H})}\ell_\Gamma[D_\Gamma^\mathcal{H}]   
 \end{equation}
for any choice of $i$, where 
\begin{itemize}
\item All enhanced level graphs $\Gamma$ that appear have two levels (and hence no horizontal nodes),
\item $D_\Gamma^\mathcal{H}$ is the divisor in $\Mbarmu{0}{n}$ associated to such an enhanced level graph $\Gamma$,
\item ${}_iLG_1(\mathcal{H})$ is the set of such enhanced level graphs with $i$th point located on level $1$, 
\item $\ell_\Gamma$ the least common multiple of the prongs $\kappa_e$ along the edges of the enhanced level graph $\Gamma$, and
\item $\psi_i$ is defined as the pullback of $\psi_i$ under the forgetful morphism to $\Mbar{0}{n}$. 
\end{itemize}

We record the class of the pushforward of this class as a lemma as it will be used repeatedly.

\begin{lemma}\label{lem:rhyp}
Let $\mu=(a_1,\dots,a_m,(-1)^{n-m})$  the \emph{resonance hyperplane class} $D^\mu=-\phi_*\xi$ in $\Pic(\Mbar{0}{n})$ is given by
$$D^\mu\equiv \sum_{i\in T\subset [n]}\max\{0,\sum_{j\in T}m_j-|T^-|+1\}\delta_T  $$
for any choice of $i=m+1,\dots,n$, where $T^-:=T\cap\{m+1,\dots,n\}$, the subset of poles in $S$. {{}}
\end{lemma}

\begin{proof}
All divisors $D_\Gamma$ with more than two vertices in $\Gamma$ are contracted by $\phi$. 
\end{proof}

\subsection{Stein factorisations}\label{sec:Stein}
We specialise now to our case of interest, though all ideas generalise to other cases. Let $\mu=(2^2,-1^4)$ then the residue map on $\Mbarmu{0}{8}$ has Stein factorisation
 \begin{equation}  
\begin{tikzcd}
\Mbarmu{0}{8}\arrow[dr, "{\pi}"] \arrow[r, "{f}"] &X \arrow[d, "{g}"] \\
&\PP^{4}
\end{tikzcd}
\end{equation}
where $f$ has connected fibres and $g$ is finite of degree two due to Theorem~\ref{thm:conncomp}(2). {{}}

Over the complement of the resonance hyperplane arrangement, $g:X\longrightarrow \PP^4$ is finite \'{e}tale of degree $2$. That is, for $V=X\setminus g^{-1}(\mathcal{A})$ and $U=\PP^4\setminus\mathcal{A}$
$$g\mid_V:V\longrightarrow U$$
is an \'{e}tale double cover which is classified by the monodromy representation
$$\tau: \pi_1(U)\longrightarrow \{\pm1\}=\mathfrak S_{2}.$$
The following lemma describes this character.{{}}

\begin{lemma}\label{lem:monodromy}
If $\gamma_S$ is a small meridian loop around the resonance hyperplane $H_S$ in $\pi_1(U)$ then
$$\tau(\gamma_S)=\begin{cases}
-1&\text{ if $|S|=1$ or $5$}\\
+1&\text{otherwise}
\end{cases}    $$
\end{lemma}

\begin{proof}
Geometrically, the lemma says that the monodromy is nontrivial precisely when the residue arrangement cannot be realised by a differential in $\HH(\mu)$, or alternatively, the pull back of $H_S$ under $\pi$ is supported on the boundary of $\Mbarmu{0}{8}$.{{}}

First consider $|S|\ne 1$ or $5$. The associated general residue configuration is realisable \cite{GT21} and the relative homology inducing period coordinates 
can be chosen locally around a general point point in the interior and the pullback $\pi^*H_S$ to include the residues and hence give a projection $\mathcal{U}\times\mathcal{V}\longrightarrow \mathcal{U}$ for contractible $\mathcal{U}$ and $\mathcal{V}$.

By the symmetry of the situation, the value of $\tau(\gamma_S)$ for $|S|=1$ or $5$ are equal for all choices of such $S$. Consider the case that $\tau(\gamma_S)=+1$. Then the multiplicity of all components of the branch divisor of $g$ are even. This implies that the normalisation $\tilde{g}:\widetilde{X}\longrightarrow \PP^4$ of $g$ consists of two copies of $\PP^4$. However, this would contradict the irreducibility of $\Mbarmu{0}{8}$. Hence the result holds.
\end{proof}

\begin{lemma}
The normalisation $\tilde{g}:\widetilde{X}\longrightarrow \PP^4$ of $g:X\longrightarrow \PP^4$ is finite of degree $2$ branched over the reduced divisor $B\in|\OO(6)|$ defined by
$$G(z_0,z_1,z_2,z_3,z_4)=z_0z_1z_2z_3z_4(z_0+z_1+z_2+z_3+z_4)=0,$$
and $\widetilde{X}$ is a normal $\QQ$-factorial Fano variety with Picard rank $\rho(\widetilde{X})=1$. 
\end{lemma}

\begin{proof}
The branch divisor of $g$ is supported on the resonance hyperplane arrangement. The normalisation of $g$ is a double cover of $\PP^4$ simply ramified over the components that appear in the branch divisor of $g$ with odd multiplicity, or equivalently, over the components for which the monodromy acts nontrivially. Hence by Lemma~\ref{lem:monodromy} the normalisation is as stated.{{}}

The branch divisor $B$ is a reduced simple normal-crossings divisor. Hence $\widetilde{X}$ has finite quotient singularities~\cite[Lemma 3.24]{EV} and is hence normal and $\QQ$-factorial. 

By Riemann--Hurwitz we have
$$K_{\widetilde{X}}= \tilde{g}^*(K_{\PP^4})+\mathfrak{ram}(\tilde{g})= \tilde{g}^*(-5H)+\frac{1}{2} \tilde{g}^*(6H)=\tilde{g}^*(-2H)   $$
then as $H$ is ample we obtain $\widetilde{X}$ is Fano, and further $\widetilde{X}$ is a hypersurface in $\PP(1,1,1,1,1,3)$ cut out by the equation
$$ w^2= z_0z_1z_2z_3z_4(z_0+z_1+z_2+z_3+z_4).  $$
Hence $\widetilde{X}\in|\mathcal{O}(6)| $ is a normal section of an ample and base point free line bundle. As weighted projective space is normal with toric and hence rational singularities and $\dim(\PP(1,1,1,1,1,3))=5$, by the Grothendieck-Lefschetz Theorem for normal projective varieties~\cite[Thm 1]{RS05} together with the Picard reduction discussed in \S1 and Theorem 2 of loc. sit., the restriction
$$\Pic( \PP(1,1,1,1,1,3))\longrightarrow \Pic(\widetilde{X})  $$
is an isomorphism and hence $\rho(\widetilde{X})=1$. 
\end{proof}

\section{On the effective cone of $\Mbar{0}{n}$}\label{sec:effcone}

\subsection{Polyhedrality.} In this section we prove the main theorems of the paper by identifying an extremal ray of the cone of moving curves at which the cone is  non-polyhedral. We begin by identifying a number of extremal rays.

\Fmuextremal*

\begin{proof}
Lemma~\ref{lem:noexcept} shows that $r_\mu:\Mbar{0}{n}\dashrightarrow \PP^{n-4}$ restricts to a morphism of the complement of the resonance hyperplane arrangement and hence Theorem~\ref{thm:exrat} applies.
\end{proof}

Now consider the resolution of $r_\mu$ by the multi-scale compactification.

\begin{proposition}\label{prop:tildeFmuextreme}
Let $\mu=(a,n-4-a,(-1)^{n-2})$ for $1\leq a\leq n-5$. Then the generic fibre of
$$\pi:\Mbarmu{0}{n}\longrightarrow \PP^{n-4}  $$
forms an extremal moving curve with class denoted by $\widetilde{F}_\mu$, and $\phi_*\widetilde{F}_\mu=F_\mu$.
\end{proposition}

\begin{proof}
This follows directly from Theorem~\ref{thm:extremal} and Lemma~\ref{lem:resolve}.
\end{proof}

Before specialising to the case of interest, we require more information on the pullback of resonance hyperplanes. Of particular interest are the components of the pullback of the resonance hyperplane arrangement that intersect the interior of $\Mbarmu{0}{n}$. 

\begin{definition}\label{def:resdiv}
For $\mu=(a_1,\dots,a_m,(-1)^{n-m})$ with $a_i>0$ and $\sum a_i=n-m-2$, the resonance transform divisor $\widetilde{H}_S$ in $\Mbarmu{0}{n}$ is defined as the components of $\pi^*H_S$ that intersect the interior and $D^\mu_S:=\phi_*\widetilde{H}_S$ in $\Mbar{0}{n}$.
\end{definition}

Recall that If $|S|=1$ or $n-m-2$ then $\pi^*H_S$ is supported on the boundary. The irreducibility of these divisors in all other cases follows directly from a recent result on the connected components of generalised strata of meromorphic differentials with residue conditions that generalises the techniques of Kontsevich and Zorich~\cite{KZ03} to this situation. 

\begin{lemma}[Thm 1.8, Thm 1.10 \cite{LW25}]\label{lem:irred}
$\widetilde{H}_S$ and hence $D^\mu_S$ is irreducible for any non-empty subset $S\subset \{0,\dots,n-m-2\}$ with $|S|\ne 1$ or $n-m-2$. 
\end{lemma}

We now specialise to the case of interest and let $\mu=(2^2,-1^6)$. We compute the rank of the space of pseudo-effective divisors in $\Mbarmu{0}{8}$ contracted by $\pi$. The result implies the extremal ray of the cone of moving curves given by the general fibre of $\pi$ is non-polyhedral. We then push these results forward to the space of interest $\Mbar{0}{8}$ to obtain non-polyhedrality of the cone of moving curves at the extremal ray $F_\mu$. Again, all techniques generalise to other cases.

\begin{proposition}\label{prop:rankup}
Let $\mu=(2^2,-1^6)$, the rank of $\RR\otimes\{D\in \Eff(\Mbarmu{0}{8}) |\pi_*D=0  \}$ is $$\rho(\Mbarmu{0}{8})-6.$$
\end{proposition}

\begin{proof}
Consider $\Pic(\Mbarmu{0}{8})=\phi^*\Pic(\Mbar{0}{8})\oplus \exc(\phi)$, where $\exc(\phi)$ is generated by the exceptional loci of $\phi$, that is,
$$ \exc(\phi)=\RR\otimes\{D_\Gamma |\phi_*(D_\Gamma)=0\}.  $$
Observe $\pi_*D=0$ for all $D\in \exc(\phi)$ by Lemma~\ref{lem:noexcept}. 

Now consider the Kapranov basis for $\Pic(\Mbar{0}{8})$ centred at the first point given by
$$ H=\psi_1\text{  and } E_S=\delta_{\{1\}\cup S}  $$
for $S\subset \{2,\dots,8\}$ and $1\leq |S|\leq 4$. Observe that $\pi_*\phi^*E_S\ne0$ only for $S=\{2\}$ and $S\subset \{3,4,5,6,7,8\}$ with $|S|=3$. In all other cases, giving the dual graph of a general curve in $E_S$ an enhanced level structure would result in at least one pole on a level below zero. As all exceptional components of $\phi$ are also contracted, $\pi_*\phi^*E_S=0$. Hence we have $22$ basis elements left and we simplify the situation by working in $\Pic(\Mbar{0}{8})/\langle E_S\mid \pi_*\phi^*E_S=0\rangle$. {{}}

The only boundary divisors in $\Mbar{0}{8}$ not appearing in the chosen basis $\Pic(\Mbar{0}{8})$ are of the form $\delta_{\{i,j\}}$ for $\{i,j\}\in\{2,\dots,8\}$. For all such divisors $\pi_*\phi^*\delta_{\{i,j\}}=0$ as the specified node necessitates a pole on a level below zero, hence imposing at least a codimension one condition on the image in $\PP^4$. The class of $\delta_{i,j}$ in the Kapranov basis is given by
$$ H-\sum_{i,j\notin S}E_S   $$
which provides $21$ divisor classes. The only remaining effective classes supported on the pullback of the resonance arrangement are the components of the pullback that intersect the interior. Lemma~\ref{lem:irred} shows for $|S|\ne1$ or $5$ there is a unique irreducible effective component intersecting the interior in the pullback $\pi^*H_S$ of the resonance hyperplane. The class of this divisor in $\Pic(\Mbarmu{0}{8})$ is hence equal to $-\xi$ minus an effective class completely supported on the boundary. Recall for $|S|=1$ or $5$ the pullback $\pi^*H_S$ is supported on the boundary. 

It follows from Lemma~\ref{lem:rhyp} that the resonance hyperplane class in $\Pic(\Mbar{0}{8})/\langle E_S\mid \pi_*\phi^*E_S=0\rangle$ is
$$D^\mu= 12H-10E_2-3\sum_{S\in\mathcal{T}} E_S   $$
where $\mathcal{T}:=\{S\subset\{3,\dots,8\}\mid |S|=3\}$. Considering the 21 boundary classes and $D^\mu$ in our basis for $\Pic(\Mbar{0}{8})/\langle E_S\mid \pi_*\phi^*E_S=0\rangle$ we have the following $22\times 22$ matrix which has rank $16$.
\tiny{$$ 
\left[\begin{array}{*{22}{r}}
  12 & -10 & -3 & -3 & -3 & -3 & -3 & -3 & -3 & -3 & -3 & -3 & -3 & -3 & -3 & -3 & -3 & -3 & -3 & -3 & -3 & -3 \\
  1 & -1 & 0 & 0 & 0 & 0 & 0 & 0 & 0 & 0 & 0 & 0 & 0 & 0 & 0 & 0 & 0 & 0 & -1 & -1 & -1 & -1 \\
  1 & -1 & 0 & 0 & 0 & 0 & 0 & 0 & 0 & 0 & 0 & 0 & 0 & 0 & 0 & -1 & -1 & -1 & 0 & 0 & 0 & -1 \\
  1 & -1 & 0 & 0 & 0 & 0 & 0 & 0 & 0 & 0 & 0 & 0 & 0 & -1 & -1 & 0 & 0 & -1 & 0 & 0 & -1 & 0 \\
  1 & -1 & 0 & 0 & 0 & 0 & 0 & 0 & 0 & 0 & 0 & 0 & -1 & 0 & -1 & 0 & -1 & 0 & 0 & -1 & 0 & 0 \\
  1 & -1 & 0 & 0 & 0 & 0 & 0 & 0 & 0 & 0 & 0 & 0 & -1 & -1 & 0 & -1 & 0 & 0 & -1 & 0 & 0 & 0 \\
  1 & -1 & 0 & 0 & 0 & 0 & 0 & 0 & 0 & -1 & -1 & -1 & 0 & 0 & 0 & 0 & 0 & 0 & 0 & 0 & 0 & -1 \\
  1 & -1 & 0 & 0 & 0 & 0 & 0 & -1 & -1 & 0 & 0 & -1 & 0 & 0 & 0 & 0 & 0 & 0 & 0 & 0 & -1 & 0 \\
  1 & -1 & 0 & 0 & 0 & 0 & -1 & 0 & -1 & 0 & -1 & 0 & 0 & 0 & 0 & 0 & 0 & 0 & 0 & -1 & 0 & 0 \\
  1 & -1 & 0 & 0 & 0 & 0 & -1 & -1 & 0 & -1 & 0 & 0 & 0 & 0 & 0 & 0 & 0 & 0 & -1 & 0 & 0 & 0 \\
  1 & -1 & 0 & 0 & -1 & -1 & 0 & 0 & 0 & 0 & 0 & -1 & 0 & 0 & 0 & 0 & 0 & -1 & 0 & 0 & 0 & 0 \\
  1 & -1 & 0 & -1 & 0 & -1 & 0 & 0 & 0 & 0 & -1 & 0 & 0 & 0 & 0 & 0 & -1 & 0 & 0 & 0 & 0 & 0 \\
  1 & -1 & 0 & -1 & -1 & 0 & 0 & 0 & 0 & -1 & 0 & 0 & 0 & 0 & 0 & -1 & 0 & 0 & 0 & 0 & 0 & 0 \\
  1 & -1 & -1 & 0 & 0 & -1 & 0 & 0 & -1 & 0 & 0 & 0 & 0 & 0 & -1 & 0 & 0 & 0 & 0 & 0 & 0 & 0 \\
  1 & -1 & -1 & 0 & -1 & 0 & 0 & -1 & 0 & 0 & 0 & 0 & 0 & -1 & 0 & 0 & 0 & 0 & 0 & 0 & 0 & 0 \\
  1 & -1 & -1 & -1 & 0 & 0 & -1 & 0 & 0 & 0 & 0 & 0 & -1 & 0 & 0 & 0 & 0 & 0 & 0 & 0 & 0 & 0 \\
  1 & 0 & 0 & 0 & 0 & 0 & 0 & 0 & 0 & 0 & 0 & 0 & -1 & -1 & -1 & -1 & -1 & -1 & -1 & -1 & -1 & -1 \\
  1 & 0 & 0 & 0 & 0 & 0 & -1 & -1 & -1 & -1 & -1 & -1 & 0 & 0 & 0 & 0 & 0 & 0 & -1 & -1 & -1 & -1 \\
  1 & 0 & 0 & -1 & -1 & -1 & 0 & 0 & 0 & -1 & -1 & -1 & 0 & 0 & 0 & -1 & -1 & -1 & 0 & 0 & 0 & -1 \\
  1 & 0 & -1 & 0 & -1 & -1 & 0 & -1 & -1 & 0 & 0 & -1 & 0 & -1 & -1 & 0 & 0 & -1 & 0 & 0 & -1 & 0 \\
  1 & 0 & -1 & -1 & 0 & -1 & -1 & 0 & -1 & 0 & -1 & 0 & -1 & 0 & -1 & 0 & -1 & 0 & 0 & -1 & 0 & 0 \\
  1 & 0 & -1 & -1 & -1 & 0 & -1 & -1 & 0 & -1 & 0 & 0 & -1 & -1 & 0 & -1 & 0 & 0 & -1 & 0 & 0 & 0 \\
\end{array}\right]
 $$}
\normalsize
That is, we have shown that the rank of the $\RR$-vector space generated by effective divisors supported in $\pi^*\mathcal{A}$, the pullback of the resonance hyperplane arrangement, has rank $\rho(\Mbarmu{0}{8})-6$. It remains now to show that any effective divisor not supported over the resonance hyperplane arrangement that lies in $\ker(\pi_*)$ has class in this subgroup. {{}}

Consider the Stein factorisation of $\pi$ restricted to $U$, the complement of the resonance arrangement in $\PP^4$.
\begin{equation}  
\begin{tikzcd}
W\arrow[dr, "{\pi}"] \arrow[r, "{f}"] &V \arrow[d, "{g}"] \\
&U
\end{tikzcd}
\end{equation}
Now $V=X\setminus g^{-1}(\mathcal{A})=\widetilde{X}\setminus \tilde{g}^{-1}(\mathcal{A})$ and $\Pic(\widetilde{X})$ is generated by $\tilde{g}^*H_S$ for any component of the resonance hyperplane arrangement. However as we have removed the preimage of this arrangement we have $\tilde{g}^*H_S$ restricts to $V$ to give the trivial line bundle and hence the Picard rank of $\rho(V)=0$.{{}}

The restricted $f:W\longrightarrow V$ is a morphism with connected irreducible one dimensional projective fibres and both varieties are $\QQ$-factorial. Hence any irreducible effective divisor on $W$ with one dimensional fibres over its image in $V$ is the pullback of an effective divisor on $V$ and such a divisor is necessarily trivial as the pullback of the trivial bundle. {{}}

Hence any irreducible effective divisor in $\Mbarmu{0}{8}$ contracted by $\pi$ has a class that can be expressed in effective classes supported over the pullback of the resonance hyperplane arrangement.
\end{proof}

This immediately extends to the pseudo-effective cone via Theorem~\ref{thm:DJV} to give the following:

\begin{proposition}\label{prop:rankuppseudo}
Let $\mu=(2^2,-1^6)$, the rank of $\RR\otimes\{D\in \overline{\Eff}(\Mbarmu{0}{8}) |\pi_*D=0  \}$ is $\rho(\Mbarmu{0}{8})-6.$
\end{proposition}

The result then pushes forward to the moduli space of interest.

\corank*
 
 \begin{proof}
The result follows from the identification $\Pic(\Mbarmu{0}{8})=\phi^*\Pic(\Mbar{0}{8})\oplus \exc(\phi)$ and $\phi_*\widetilde{F}_\mu=F_\mu$ while $
\widetilde{F}_\mu\cdot D_\Gamma=0$ for all $D_\Gamma\in  \exc(\phi)$ via Lemma~\ref{lem:noexcept}.
 \end{proof}

This provides the main theorem of this paper.

\MOn*

\begin{proof}
For $n=8$ and $\mu=(2^2,-1^6)$ the curve class $F_\mu$ forms an extremal moving curve by Proposition~\ref{thm:Fmuextremal}. By Theorem~\ref{thm:corank} the cone of moving curves $\movcurve(\Mbar{0}{8})$ is not polyhedral at this ray. The forgetful morphism extends this result to higher $n$.
\end{proof}

\MDS*

\begin{proof}
A finitely generated Cox ring implies a polyhedral effective cone, hence the result follows from Theorem~\ref{thm:nonpoly}.
\end{proof}

\subsection{Resonance transform divisors.} Our investigation has identified a number of interesting divisor classes in $\Mbar{0}{n}$. These are the divisors that are obtained from components of the pullback of the resonance hyperplane arrangement defined in Definition~\ref{def:resdiv}. In this section we show that these are in fact rigid and extremal divisors. We begin by considering the case where $m=2$.

\begin{lemma}
For a fixed $\mu=(a,n-4-a,(-1)^{n-2})$ for $1\leq a\leq n-5$ the fibre of $\pi$ over a point lying in exactly one resonance hyperplane is one dimensional. 
\end{lemma}

\begin{proof}
The intersection of the fibre with the interior is clearly one dimensional and we are left to consider the boundary. No exceptional divisor dominate $\PP^{n-4}$, by Lemma~\ref{lem:noexcept} hence if any dominate the resonance hyperplane divisor in $\PP^{n-4}$ they must have general fibres of dimension 1.
\end{proof}

This implies that the fibre over a general point in a resonance hyperplane has the same class as the general fibre. There is a natural decomposition of the fibre over a general point in a resonance hyperplane as a sum of components that intersect the interior and those that are supported on the boundary, that is,
$$\widetilde{F}_\mu=\widetilde{B}_S +\widetilde{B}_S^\Delta .  $$
Negative intersection with irreducible curves covering a Zariski dense subset of an irreducible effective divisor is a well-known trick for proving rigidity and extremality~(See for example~\cite[Lemma 4.1]{ChenCoskun}).

\begin{proposition}
Let $\mu=(a,n-4-a,(-1)^{n-2})$ for $1\leq a\leq n-5$, the divisor $\widetilde{H}_S$ in $\Mbarmu{0}{n}$ is rigid and extremal.
\end{proposition}

\begin{proof}
$\widetilde{B}_S$ forms a covering curve for $\widetilde{H}_S$, that is, curves with class equal to $\widetilde{B}_S$ cover a Zariski dense subset of $\widetilde{H}_S$. Further, $\widetilde{F}_\mu\cdot\widetilde{H}_S=0$. Consider the component of $\widetilde{B}_S^\Delta$ specified by the enhanced level graph with three vertices, two on top level labelled with the marked points corresponding to the poles with one vertex containing the markings from $S$ and the other vertex labelled with the remaining poles, connected to one vertex on level $-1$ labelled with the zeros (points $p_1$ and $p_2$). 

Consider a multi-scale differentials of this type such that the residues at the poles of $\eta_{-1}$, the differential on level $-1$, corresponding to the two edges are both zero. By \cite[Prop 4.2]{CMZ22}, such a differential arises as the limit of differentials in the smooth locus of $\widetilde{H}_S$ and hence $\widetilde{B}_S^\Delta\cdot \widetilde{H}_S>0$ which implies $\widetilde{B}_S\cdot \widetilde{H}_S<0$.
\end{proof}

Irreducible curves covering a Zariski dense subset of an irreducible divisor can also be used to identify rigid components of an effective divisor with multiple components (See for example~\cite[Lemma 2.1]{Mnonpoly} and surrounding discussion). We use this strategy to show that $\phi^*D_{S}^\mu$ and hence $D_S^\mu$ is rigid.

\begin{proposition}\label{prop:m2}
Let $\mu=(a,n-4-a,(-1)^{n-2})$ for $1\leq a\leq n-5$, the divisor $\phi_*\widetilde{H}_S=D^\mu_S$ in $\Mbar{0}{n}$ is rigid and extremal.
\end{proposition}

\begin{proof}
We have
$$\phi^*D^\mu_S=\phi^*\phi_*\widetilde{H}_S= \widetilde{H}_S+\sum a_{\Gamma,S}D_\Gamma  $$
for non-negative $a_{\Gamma,S}$ where these constants are zero for $\Gamma$ such that $D_\Gamma$ is not contracted by $\phi$ or the image of $D_\Gamma$ under $\phi$ is not supported on $D^\mu_S$. Observe now that $\widetilde{B}_S\cdot D_\Gamma=0$ unless all poles appear on the top level of $\Gamma$ across two vertices partitioned as $S$ and the remaining poles. This is the only way that the residues can realise a general point of $H_S$.{{}}

We are left to consider the ways that the zeros (points $p_1$ and $p_2$) can be distributed across the possible level graphs of this type. Now $\widetilde{B}_S\cdot D_\Gamma=0$ if both points are not on the lower level vertex as for the multi-scale differentiable to be smoothable with the extra condition on the residues, the differential on the lower level will require zero residues at the poles corresponding to the edges of the dual graph by \cite[Prop 4.2]{CMZ22} providing a contradiction.{{}}

Hence we have just one $\Gamma$ to consider with points $p_1$ and $p_2$ on the vertex on lower level. However,  \cite[Prop 4.2]{CMZ22} in this case implies that the differential on lower level will require zero residues at both poles corresponding to the edges of the level graph. This is a non-trivial condition on the underlying stable curve and hence the image of $D_\Gamma$ is not supported on $D^\mu_S$. Hence $a_{\Gamma,S}=0$

Hence $\widetilde{B}_S\cdot\phi^*D^\mu_S=\widetilde{B}_S\cdot \widetilde{H}_S$ for the associated linear systems we obtain
$$|k\phi^*D^\mu_S|=  k\widetilde{H}_S+|k\sum a_{\Gamma,S}D_\Gamma|.   $$
However, as $a_{\Gamma,S}>0$ only for $D_\Gamma$ contracted by $\phi$ we have $\sum a_{\Gamma,S}D_\Gamma$ is rigid, implying $\phi^*D^\mu_S$ and hence $D^\mu_S$ are rigid. As $D^\mu_S$ is irreducible this implies extremality.
\end{proof}

We now provide an inductive argument to prove extremality of the resonance transform divisors for $m\geq3$.

\extremaldivisors*

\begin{proof}
Let $m\geq 3$, we have a natural short exact sequence
$$0\longrightarrow \OO_{\Mbar{0}{n}}(kD_S^\mu-\delta_{\{1,2\}})\longrightarrow   \OO_{\Mbar{0}{n}}(kD_S^\mu)\longrightarrow\OO_{\delta_{\{1,2\}}}(kD_S^\mu) \longrightarrow 0    $$
However, under the identification $\delta_{\{1,2\}}\cong \Mbar{0}{n-1}$ we have $D_S^\mu\mid_{\delta_{\{1,2\}}}=D_S^{\mu'}$ where $\mu'=(a_1+a_2,a_3,\dots,a_m,(-1)^{n-m})$. Hence the long exact sequence in cohomology yields the exact sequence
$$H^0(kD_S^\mu-\delta_{\{1,2\}})\longrightarrow H^0(kD_S^\mu)\longrightarrow H^0(kD_S^{\mu'}).   $$
Then $\dim H^0(kD_S^{\mu'})=1$ by the inductive hypothesis and it is sufficient to show $H^0(kD_S^\mu-\delta_{\{1,2\}})=0$.{{}}

Consider $\widetilde{F}_\mu$ in $\Mbarmu{0}{n}$ the fibre of a general point, that is, a point in the complement of the resonance hyperplane arrangement in $\PP^{n-m-2}$. Then $\widetilde{F}_\mu$ is an $(m-1)$-dimensional subvariety of $\Mbarmu{0}{n}$ and we note that $\phi^*D_S^\mu\mid_{\widetilde{F}_\mu}$ is trivial as $\widetilde{H}_S\mid_{\widetilde{F}_\mu}$ is trivial and any divisor $D_\Gamma$ such that $D_\Gamma\mid_{\widetilde{F}_\mu}$ is non-trivial will require all poles to be on the same vertex and on top level. The image of such $D_\Gamma$ is not supported on $D_S^\mu$.{{}}

In contrast, $\phi^*\delta_{\{1,2\}}$ restricts to be effective on ${\widetilde{F}_\mu}$. We define the curve $B:={\widetilde{F}_\mu}\cdot A^{m-2}$ for any ample divisor $A$. As ${\widetilde{F}_\mu}$ is general, curves with the class $B$ cover a Zariski dense subset of $\Mbarmu{0}{n}$ and hence $B$ is a moving curve and has non-negative intersection with any pseudo-effective divisor in $\Mbarmu{0}{n}$. {{}}

However $ B\cdot \phi^*\delta_{\{1,2\}} >0$ and $B\cdot \phi^*D_S^\mu=0$ hence $B\cdot(\phi^*(kD^\mu_S-\delta_{\{1,2\}}))<0$ which implies that the divisor is not pseudo-effective and hence $H^0(kD_S^\mu-\delta_{\{1,2\}})=0$.{{}}

This concludes the inductive step and the proof follows from the $m=2$ case proven in Proposition~\ref{prop:m2}.
\end{proof}

To complete the paper, we provide the classes of these extremal divisors.

\begin{proposition}\label{prop:divclass}
For $\mu=(a_1,\dots,a_m,(-1)^{n-m})$ with $m\geq2$, $a_i>0$ and $\sum a_i=n-m-2$, the resonance transform divisor $D_S^\mu$ for $S\subset\{m+1,\dots,n-1\}$ has class
$$ D_S^\mu= D^\mu-\Delta_{\widetilde{S}}         $$
where $D^\mu$ is the resonance hyperplane class and 
$$\Delta_{\widetilde{S}} = \sum_T c_S(T)\cdot\delta_T$$
with 
$$c_S(T)=\begin{cases}|a_T -|T^-|+1| &\text{ for $T^-=S$}\\
\max\{0,a_T-|T^-|+1\}&\text{ for $S\subsetneqq T^-$ }\\
\max\{0, |T^-|-1-a_T\}&\text{ for $T^-\subsetneqq S$ }\\
0&\text{otherwise} \end{cases}    $$
for $T^-:=T\cap \{m+1,\dots,n\}$, $T^+:=T\cap \{1,\dots,m\}$ and $a_T=\sum_{j\in T^+}a_j$.
\end{proposition}

\begin{proof}
Pulling back $H_S$ from $\PP^{n-m-2}$ we obtain
$$\pi^*\OO(1)=\pi^*H_S=\widetilde{H}_S+\sum_{\Gamma\in\mathcal{T}}D_\Gamma$$
where $\mathcal{T}$ is the collection of level graphs $\Gamma$ such that the image of $D_\Gamma$ is supported on $H_S$. Transversality is clear and implies all such divisors $D_\Gamma$ appear with multiplicity one. We now observe that 
$$ \phi_*D_\Gamma=\begin{cases} \kappa_e\delta_S&\text{ if $\Gamma$ has a unique vertical edge}\\
 \delta_S&\text{ if $\Gamma$ has a unique horizontal edge}\\
0&\text{otherwise.} \end{cases}  $$
If $\Gamma$ has more than one edge it is contracted by $\phi$, the multiplicity $\kappa_e$ counts the number of prong matchings in the case of $\Gamma$ having a unique vertical edge. We refer the reader to \cite{BCGGM} for a discussion of prong matchings. In the case that $\Gamma$ has a unique horizontal edge there is a unique prong matching, though this case will play no role as no such $\Gamma$ appears in $\mathcal{T}$.

Let $\Gamma$ be a two vertex, two level graph such that $D_\Gamma$ has image $\delta_T$ in $\Mbar{0}{n}$ for $n\notin T$. Then the image of $D_\Gamma$ lies in $H_S$ (and hence $\Gamma\in \mathcal{T}$) in three cases:
\begin{enumerate}
\item
$T^-=S$, with the markings of $T$ appearing on either upper or lower level of $\Gamma$, 
\item
$S\subsetneqq T^-$ and the markings of $T$ appear on the lower level of $\Gamma$,
\item
 $T^-\subsetneqq S$ and the markings of $T$ appear on upper level of $\Gamma$.
\end{enumerate}
The specification of $c_S(T)$ computes $\kappa_e$ in these cases and returns zero otherwise.  Hence we obtain
$$\phi_*\widetilde{H}_S=\phi_*\pi^*\OO(1)-\sum_{\Gamma\in\mathcal{T}}\phi_*D_\Gamma$$
which returns the required formula.
\end{proof}


\end{document}